\documentclass[acmtoms]{acmtrans2m}

\usepackage{graphicx}
\usepackage{psfrag}
\usepackage{url}
\usepackage{amsmath}
\usepackage{amssymb}
\usepackage{algorithm}
\usepackage{fancyvrb}
\usepackage{color}

\newcommand{\dx}{\, \mathrm{d}x}
\newcommand{\dX}{\, \mathrm{d}X}
\newcommand{\R}{\mathbb{R}}

\newcommand{\tab}{\hspace*{2em}}

\newtheorem{theorem}{Theorem}
\newtheorem{lemma}{Lemma}

\DefineVerbatimEnvironment{code}{Verbatim}{frame=single,rulecolor=\color{blue}}

\title{Efficient Compilation of a Class of \\ Variational Forms}

\author{Robert C. Kirby \\ The University of Chicago
        \and
	Anders Logg \\ Simula Research Laboratory}

\begin{abstract}
  We investigate the compilation of general multilinear variational
  forms over affines simplices and prove a representation theorem for
  the representation of the element tensor (element stiffness matrix)
  as the contraction of a constant reference tensor and a geometry
  tensor that accounts for geometry and variable coefficients. Based
  on this representation theorem, we design an algorithm for efficient
  pretabulation of the reference tensor.  The new algorithm has been
  implemented in the FEniCS Form Compiler~(FFC) and improves on a
  previous loop-based implementation by several orders of magnitude,
  thus shortening compile-times and development cycles for users of
  FFC.
\end{abstract}




\begin{document}

\begin{bottomstuff}
Robert C. Kirby, Department of Computer Science, University of Chicago,
1100 East 58th Street, Chicago, Illinois 60637, USA.
\emph{Email:} \texttt{kirby@cs.uchicago.edu}.
This work was supported by the United States Department of Energy
under grant DE-FG02-04ER25650.
\newline
Anders Logg,
Simula Research Laboratory,
Martin Linges v 17, Fornebu,
PO Box 134, 1325 Lysaker, Norway.
\emph{Email:} \texttt{logg@simula.no}.
\end{bottomstuff}

\maketitle

%------------------------------------------------------------------------------
\section{Introduction}

It is our goal to improve the efficiency of compiling variational
forms with FFC, the FEniCS Form Compiler, previously presented
in~\cite{logg:article:10}. FFC automatically generates efficient
low-level code for evaluating a wide class of multilinear variational
forms associated with finite element
methods~\cite{Cia76,Hug87,BreSco94,EriEst96} for partial differential
equations. However, the efficiency of FFC decreases rapidly with the
complexity of the variational form and polynomial degree. In this
paper, we investigate the core algorithms of FFC and rephrase them so
as to diminish interpretive overhead and make better use of optimized
numerical libraries. We thus wish to decrease the run-time
for the form compiler, corresponding to a reduced compile-time for a
finite element code. This becomes
particularly important when FFC is used as a just-in-time compiler to
generate and compile code on the fly at run-time.

\subsection{FFC, the FEniCS Form Compiler}

FFC~\cite{logg:www:04} was first released in 2004 as a prototype
compiler for variational forms, automating a key step in the
implementation of finite element methods.~\cite{logg:thesis:03}.
Given a multilinear variational form and an affinely mapped simplex,
FFC automatically generates low-level code for evaluation of the
variational form (assembly of the associated linear system). More
precisely, FFC generates efficient low-level code for computation of
the element tensor (element stiffness matrix), based on the novel
approach of representing the element tensor as a special tensor
contraction presented earlier in~\cite[SISC]{logg:article:07}
and~\cite{KirKne04}. FFC is implemented in Python and provides
both a command-line and Python interface for the specification of
variational forms in a syntax very close to the mathematical notation.
Together with other components of the FEniCS
project~\cite[FEniCS]{logg:www:03}, such as
FIAT~\cite{www:FIAT,Kir04,Kir05} and
DOLFIN~\cite[DOLFIN]{logg:www:01}, FFC automates some central aspects
of the finite element method.

There exist today a number of competing efforts that strive to
automate the finite element method. One such example is
Sundance~\cite{Lon03}, which similarly to FFC provides a system for
automated assembly/evaluation of variational forms given in mathematical
notation. A main difference between Sundance and FFC is that Sundance
provides a run-time system for parsing and evaluation of variational
forms, whereas FFC precomputes important quantities at compile-time,
which in many cases allows for generation of more efficient code for
the run-time assembly of linear systems. Automated assembly from a
high-level specification of a variational form is also supported by
FreeFEM~\cite{www:FreeFEM}, GetDP~\cite{www:GetDP} and
Analysa~\cite{www:Analysa}, which all implement domain-specific
languages for specification and implementation of finite element
methods for partial differential equations. Other projects such as
Diffpack~\cite{Lan99} and deal.II~\cite{www:deal.II} provide
sophisticated libraries aiding implementation of finite element
methods.  These libraries provide tools such as meshes, ordering of
degrees of freedom, and interfaces to solvers, but do not provide
automated evaluation of variational forms.

Since the first release of FFC, a number of improvements have been
made, mostly improving on the functionality of the compiler. New
features that have been added since the work~\cite{logg:article:10}
include support for mixed finite element formulations, an extension of
the form language to include linear algebra operations such as inner
products and matrix-vector products, differential operators such as
the gradient, divergence and rotation, local projections between
finite element spaces and an option to generate code in terms of level
2 BLAS operations. FFC also functions as a just-in-time compiler for
PyDOLFIN, the Python interface of DOLFIN~\cite[DOLFIN]{logg:www:01}.

However, the performance of FFC has been suboptimal, potentially
lengthening development cycles for high-order simulations in three
dimensions.  Additionally, this inefficiency would inhibit the
embedding of FFC in a run-time system, such as Sundance, as a
just-in-time compiler.

\subsection{Main results}
The main purpose of this paper is twofold.  First, we extend and
formalize our particular representation of multilinear variational
forms.  This involves writing each form as a sum of {\em monomials},
which are integrals of products of (derivatives of) the basis
functions.  Hence, we prove a representation theorem showing that the
evaluation of any monomial form is equivalent to a contraction of a
\emph{reference tensor} with a \emph{geometry tensor}. This makes
precise what we have discussed in previous
work~\cite{logg:article:07,logg:article:10}. Second, as evaluating the
reference tensor is a dominant cost of FFC, we discuss how to improve
the efficiency of building the reference tensor for each monomial by
rewriting the computation in terms of operations that may be performed
by optimized libraries and by hoisting loop invariants. The loop
hoisting is non-trivial since the depth of the loop nesting is not
known until the code is executed. The speedups gained with the new
algorithm for a series of test cases range between one and three
orders of magnitude.  Furthermore, we introduce the concept of
\emph{signatures} for monomial forms to allow factorization of common
monomial terms when evaluating a given multilinear form.

\subsection{Outline of this paper}

In Section~\ref{sec:evaluation}, we first derive a representation for
the element tensor as a contraction of two tensors, which is at the
core of the implementation of FFC. We then, in
Section~\ref{sec:algorithm}, discuss different approaches to
precomputation of the monomial integrals that appear in this tensor
representation, concluding that a suitable rearrangement of the
computation can lead to significant improvement in performance.

To test the new algorithm, we present in Section~\ref{sec:benchmarks}
benchmark results for a series of test cases, comparing the latest
version of FFC with a previous version. Finally, we summarize our
findings in Section~\ref{sec:conclusion}.

%------------------------------------------------------------------------------
\section{Evaluation of multilinear forms}
\label{sec:evaluation}

We review here the basic idea of tensor representation of multilinear
variational forms, as first presented in
\cite{logg:article:07,logg:article:10}, and derive a representation
theorem for a general class of multilinear forms.

At the core of finite element methods for partial differential
equations is the assembly of a linear system from a given bilinear
form. In general, we consider a multilinear form
\begin{equation}
  a : V_1 \times V_2 \times \cdots \times V_r \rightarrow \R,
\end{equation}
defined on the product space $V_1 \times V_2 \times \cdots \times V_r$
of a given set $\{V_i\}_{i=1}^r$ of discrete function spaces on a
triangulation $\mathcal{T}$ of a domain $\Omega \subset \R^d$.

Typically, $r = 1$ (linear form) or $r = 2$ (bilinear form), but the
form compiler FFC can handle multilinear forms of arbitrary
\emph{arity} $r$. In the simplest case, all function spaces are equal
but there are many important examples, such as mixed methods, where it
is important to consider arguments coming from different function
spaces.

\subsection{The element tensor}

Let
$\{\phi_i^1\}_{i=1}^{N_1},
 \{\phi_i^2\}_{i=1}^{N_2}, \ldots,
 \{\phi_i^r\}_{i=1}^{N_r}$
be bases of $V_1, V_2, \ldots, V_r$ and let $i = (i_1, i_2, \ldots,
i_r)$ be a multiindex. The multilinear form $a$ then
defines a rank $r$ tensor, given by
\begin{equation}
  A_i = a(\phi_{i_1}^1, \phi_{i_2}^2, \ldots, \phi_{i_r}^r).
\end{equation}
In the case of a bilinear form, the tensor $A$ is a matrix (the
stiffness matrix), and in the case of a linear form, the tensor $A$ is
a vector (the load vector).

As discussed in \cite{logg:article:10}, to compute the tensor $A$ by
assembly, we need to compute the \emph{element tensor} $A^K$ on each
element $K$ of the triangulation $\mathcal{T}$ of $\Omega$. With
$\{\phi^{K,1}_i\}_{i=1}^{n_1}$ the restriction to $K$ of the subset
of $\{\phi_i^1\}_{i=1}^{N_1}$ supported on $K$,
$V_1^K = \mathrm{span} \{\phi^{K,1}_i\}_{i=1}^{n_1}$
and the local spaces $V_2^K, \ldots, V_r^K$ defined similarly, we need to
evaluate the rank $r$ \emph{element tensor} $A^K$, given by
\begin{equation}
  A^K_i = a_K(\phi_{i_1}^{K,1}, \phi_{i_2}^{K,2}, \ldots,
  \phi_{i_r}^{K,r}) \quad \forall i \in \mathcal{I},
\end{equation}
where $a_K$ is the local contribution to the given multilinear
form~$a$ on the element~$K$ and where $\mathcal{I}$ is the index set
\begin{equation}
  \mathcal{I} =
  \prod_{j=1}^r[1,|V_j^K|] = \{(1,1,\ldots,1), (1,1,\ldots,2), \ldots,
  (n_1,n_2,\ldots,n_r)\}.
\end{equation}

We restrict our discussion to multilinear forms that may be written as
a sum over terms consisting of integrals over $\Omega$ of products of
derivatives of functions from sets of discrete spaces
$\{V_i\}_{i=1}^r$.  We call such terms {\em monomials}.  For one such
term, the element tensor takes the following (preliminary) form:
\begin{equation} \label{eq:preliminary}
  A^K_i = \int_K
  \prod_{j=1}^r D_x^{\delta_j} \phi^{K,j}_{\iota_j(i)} \dx,
\end{equation}
where the subscript $\iota_j(i)$ picks out a basis function from the
restriction of $V_j$ to $K$ for the current multiindex $i$ and where
$\delta_j$ is the multiindex for the corresponding derivative.
For sequences of multiindices such as $\{\delta_j\}_{j=1}^r$, we use
the convention that $\delta_{jk}$ denotes the $k$th element of the
$j$th multiindex for $k=1,2,\ldots,|\delta_j|$.

To explain the notation, we consider a couple of illustrative
examples. First, consider the bilinear form $a(v, u) = \int_{\Omega} v
u \dx$ corresponding to a mass matrix. The element tensor (matrix) is
then given by
\begin{equation}
  A^K_i = \int_K \phi_{i_1}^{K,1} \phi_{i_2}^{K,2} \dx,
\end{equation}
and so, in the notation of~(\ref{eq:preliminary}), we
have $r = 2$, $\iota_j(i) = i_j$ and $\delta_j = \emptyset$ for
$j=1,2$, where $\emptyset$ denotes an empty multiindex (basis function
is not differentiated).

Next, we consider the bilinear form
$a(v, u) = \int_{\Omega} v
\frac{\partial^2 u}{\partial x_1 \partial x_2} \dx$ with corresponding
element tensor given by
\begin{equation}
  A^K_i = \int_K \phi_{i_1}^{K,1} \frac{\partial^2
  \phi_{i_2}^{K,2}}{\partial x_1 \partial x_2} \dx,
\end{equation}
which we can phrase in the notation of~(\ref{eq:preliminary})
with $r = 2$, $\iota_j(i) = i_j$, $\delta_1 = \emptyset$
and $\delta_2 = (1, 1)$.

More generally, variational problems involve sums over monomial terms,
each of which may include a spatially varying coefficient.  We express
such a coefficient in a finite element basis as $\sum_{\gamma=1}^n
c_{\gamma} \phi_{\gamma}$. Also, the function spaces involved may each
be vector-valued.  While we might reduce this situation to a
collection of cases of the form (\ref{eq:preliminary}), we instead
extend our canonical form.
This allows us to write
the element tensor as
\begin{equation} \label{eq:canonical}
  A^K_i =
  \sum_{\gamma\in\mathcal{C}}
  \int_K
  \prod_{j=1}^m
  c_j(\gamma)
  D_x^{\delta_j(\gamma)}
  \phi^{K,j}_{\iota_j(i,\gamma)}[\kappa_j(\gamma)] \dx,
\end{equation}
with summation over some appropriate index set $\mathcal{C}$, where
$\kappa_j(\gamma)$ denotes a component index for factor~$j$ depending on~$\gamma$.
To distinguish a component index from a basis function index, we here use
$[\cdot]$ to denote a component index. Note that the number of
factors~$m$ may be different from the rank~$r$ of the element tensor
(arity of the form).

To illustrate the notation, we consider the
bilinear\footnote{Note that we may alternatively consider
this to be a \emph{trilinear form}, if we think of the coefficient~$w$ as a free
argument to the form and not as a fixed given function.}
form on $V_1 \times V_2$ for the weighted vector-valued Poisson's equation
with given variable coefficient function $w \in V_3$:
\begin{equation}
  a(v, u)
  = \int_{\Omega} w \nabla v : \nabla u \dx
  =
  \sum_{\gamma_1=1}^d \sum_{\gamma_2=1}^d
  \int_{\Omega} w
  \frac{\partial v[\gamma_1]}{\partial x_{\gamma_2}}
  \frac{\partial u[\gamma_1]}{\partial x_{\gamma_2}} \dx.
\end{equation}
The corresponding element tensor is given by
\begin{equation}
  A^K_i =
  \sum_{\gamma_1=1}^d \sum_{\gamma_2=1}^d \sum_{\gamma_3=1}^{|V_3^K|}
  \int_{\Omega}
  \frac{\partial \phi^{K,1}_{i_1}[\gamma_1]}{\partial x_{\gamma_2}}
  \frac{\partial \phi^{K,2}_{i_2}[\gamma_1]}{\partial x_{\gamma_2}}
  w^K_{\gamma_3}\phi^{K,3}_{\gamma_3} \dx,
\end{equation}
with $\{w^K_{\gamma_3}\}_{\gamma_3=1}^{|V_3^K|}$ the expansion
coefficients for $w$ in the
local basis on $K$ for $V_3$.
In the notation of (\ref{eq:canonical}),
we thus have $r = 2$, $m = 3$,
$\iota(i, \gamma) = (i_1, i_2, \gamma_3)$,
$\delta(\gamma) = (\gamma_2, \gamma_2, \emptyset)$,
$\kappa(\gamma) = (\gamma_1, \gamma_1, \emptyset)$ and
$c_j(\gamma) = (1, 1, w^K_{\gamma_3})$.
The index set $\mathcal{C}$ is given by
$\mathcal{C} = \{(\gamma_1,\gamma_2,\gamma_3) \} =
[1,d]^2 \times [1,|V_3^K|]$.

One may argue that the canonical form
(\ref{eq:canonical}) may always be reduced to the
simpler form (\ref{eq:preliminary}) by considering any
given element tensor as a suitable transformation (summation and
multiplication with coefficients) of basic element tensors of the
simple form (\ref{eq:preliminary}).  However, we shall
consider the more general form (\ref{eq:canonical})
since it increases the granularity of the operation of computing the
reference element tensor, which is the operation we set out to
optimize. It is also a more flexible representation in that it allows
us to directly express the element tensor for a wider range of
multilinear forms such as that for the vector-valued Poisson's
equation.

\subsection{Representing the element tensor as a tensor contraction}

We may write the element tensor for any (affine) element $K$ as a
contraction of one tensor depending only on the form and function spaces
with another depending only on the geometry and coefficients in the
problem. This is accomplished by affinely transforming from $K$ to a
reference element $K_0$. In order to show this, we first state some
basic results providing a notational framework for our representation
theorem.

We first make the following observation
about interchanging product and summation.
\begin{lemma}[(Interchanging product and summation)] \label{lem:changeorder}
  With $\mathcal{J}_1, \mathcal{J}_2, \ldots, \mathcal{J}_m$
  a given sequence of index sets, product and summation
  may be interchanged as follows:
  \begin{equation} \label{eq:changeorder}
    \prod_{i=1}^m \sum_{j\in\mathcal{J}_i} a_{ij} =
    \sum_{j \in \mathcal{J}_1 \times \ldots \times \mathcal{J}_m}
    \prod_{i=1}^m a_{ij_i} =
    \sum_{j \in \prod_{k=1}^m \mathcal{J}_k}
    \prod_{i=1}^m a_{ij_i},
  \end{equation}
  where each $j \in \prod_{k=1}^m \mathcal{J}_k$ is a multiindex
  of length $|j| = m$.
\end{lemma}

Using the notation of Lemma~\ref{lem:changeorder}, we may also
prove the following chain rule for higher order partial derivatives.
\begin{lemma}[(Chain rule)] \label{lem:chainrule}
  If $F : \R^d \rightarrow \R^d$ is a bijective and differentiable
  mapping (a diffeomorphism) between two coordinate systems,
  $x = F(X)$, then
  \begin{equation}
      D_x^{\delta}
      = \sum_{\delta' \in [1,d]^{|\delta|}}
      \left
      (\prod_{k=1}^{|\delta|} \frac{\partial X_{\delta'_k}}{\partial x_{\delta_k}}
      \right)
      D_X^{\delta'},
  \end{equation}
  for each multiindex $\delta$,
  where
  $D_x^{\delta} = \prod_{i=1}^{|\delta|}\frac{\partial}{\partial x_{\delta_i}}$
  and
  $D_X^{\delta'} = \prod_{i=1}^{|\delta'|}\frac{\partial}{\partial X_{\delta'_i}}$.
\end{lemma}
\begin{proof}
  By the standard chain rule and Lemma~\ref{lem:changeorder},
  we have
  \begin{equation}
    \begin{split}
      D_x^{\delta}
      &= \prod_{i=1}^{|\delta|}
      \frac{\partial}{\partial x_{\delta_i}}
      = \prod_{i=1}^{|\delta|}
      \sum_{\delta'=1}^d
      \frac{\partial X_{\delta'}}{\partial x_{\delta_i}}
      \frac{\partial}{\partial X_{\delta'}}
      = \sum_{\delta' \in [1,d]^{|\delta|}}
      \prod_{i=1}^{|\delta|}
      \frac{\partial X_{\delta'_i}}{\partial x_{\delta_i}}
      \frac{\partial}{\partial X_{\delta'_i}} \\
      &= \sum_{\delta' \in [1,d]^{|\delta|}}
      \prod_{k=1}^{|\delta|} \frac{\partial X_{\delta'_k}}{\partial x_{\delta_k}}
      \prod_{i=1}^{|\delta|}
      \frac{\partial}{\partial X_{\delta'_i}}
      = \sum_{\delta' \in [1,d]^{|\delta|}}
      \left
      (\prod_{k=1}^{|\delta|} \frac{\partial X_{\delta'_k}}{\partial x_{\delta_k}}
      \right)
      D_X^{\delta'}.
    \end{split}
  \end{equation}
\end{proof}

We may now prove the following representation theorem\footnote{A
similar representation was derived and presented in
\cite{logg:article:10} but in less formal notation.}  for the element
tensor.
\begin{theorem}[(Representation theorem)]
  If $F_K$ is a given affine mapping from a reference element $K_0$ to
  an element $K$ and $\{V_j^K\}_{j=1}^m$ is a given set of discrete
  function spaces on $K$, each generated by a discrete function space
  on the reference element through the affine mapping, that is, for
  each $\phi \in V_j^K$ there is some $\Phi \in V_j^0$ such that $\Phi
  = \phi \circ F_K$, then the element tensor~\emph{(\ref{eq:canonical})} may
  be represented as the tensor contraction of a \emph{reference
  tensor} $A^0$ and a \emph{geometry tensor} $G_K$,
  \begin{equation} \label{eq:representation}
    A^K = A^0 : G_K,
  \end{equation}
  that is,
  \begin{equation}
    A^K_i = \sum_{\alpha\in\mathcal{A}} A^0_{i\alpha} G_K^{\alpha}
    \quad \forall i \in \mathcal{I},
  \end{equation}
  where the reference tensor $A^0$ is independent of $K$. In
  particular, the reference tensor $A^0$ is given by
  \begin{equation}
    A^0_{i\alpha}
    =
    \sum_{\beta\in\mathcal{B}}
    \int_{K_0}
    \prod_{j=1}^m
    D_X^{\delta'_j(\alpha,\beta)}
    \Phi^j_{\iota_j(i,\alpha,\beta)}[\kappa_j(\alpha,\beta)]
    \dX,
  \end{equation}
  and the geometry tensor $G_K$ is the outer product of the
  coefficients of any weight functions with a tensor that depends only
  on the Jacobian $F_K$,
  \begin{equation}
    G_K^{\alpha}
    =
    \prod_{j=1}^m
    c_{j}(\alpha) \,
    \det F_K'
    \sum_{\beta\in\mathcal{B'}}
    \prod_{j'=1}^m
    \prod_{k=1}^{|\delta_{j'}(\alpha,\beta)|}
    \frac{\partial X_{\delta'_{j'k}(\alpha,\beta)}}{\partial x_{\delta_{j'k}(\alpha,\beta)}},
  \end{equation}
  for some appropriate index sets $\mathcal{A}$, $\mathcal{B}$ and
  $\mathcal{B}'$. We refer the the index set~$\mathcal{I}$ as
  the set of \emph{primary indices}, the index set~$\mathcal{A}$ as
  the set of \emph{secondary indices}, and to the index sets~$\mathcal{B}$ and
  $\mathcal{B}'$ as \emph{auxiliary indices}.
\end{theorem}
\begin{proof}
  Starting from~(\ref{eq:canonical}), we may move the product of
  constant expansion coefficients outside of the integral to obtain
  \begin{equation} \label{eq:proof,1}
    \begin{split}
      A^K_i
      &=
      \sum_{\gamma\in\mathcal{C}}
      \int_K
      \prod_{j=1}^m
      c_j(\gamma)
      D_x^{\delta_j(\gamma)}
      \phi^{K,j}_{\iota_j(i,\gamma)}[\kappa_j(\gamma)] \dx \\
      &=
      \sum_{\gamma\in\mathcal{C}}
      \prod_{j'=1}^m
      c_{j'}(\gamma)
      \int_K
      \prod_{j=1}^m
      D_x^{\delta_j(\gamma)}
      \phi^{K,j}_{\iota_j(i,\gamma)}[\kappa_j(\gamma)] \dx.
    \end{split}
  \end{equation}
  We now make a change of variables through $F_K$,
  mapping coordinates $X \in K_0$ to coordinates $x = F_K(X) \in K$,
  to carry out the integration on the reference element~$K_0$.
  By Lemma~\ref{lem:chainrule}, we thus obtain
  \begin{equation}
    \begin{split}
      A^K_i
      &=
      \sum_{\gamma\in\mathcal{C}}
      \prod_{j'=1}^m
      c_{j'}(\gamma)
      \int_{K_0}
      \prod_{j=1}^m
      \sum_{\delta' \in [1,d]^{|\delta_j|}}
      \prod_{k=1}^{|\delta_j|} \frac{\partial X_{\delta'_k}}{\partial x_{\delta_{jk}}}
      \quad\times \\
      & \quad \times \quad
      D_X^{\delta'} \Phi^j_{\iota_j(i,\gamma)}[\kappa_j(\gamma)]\
      \det F_K' \dX \\
      &=
      \sum_{\gamma\in\mathcal{C}}
      \sum_{\delta' \in \prod_{l=1}^m [1,d]^{|\delta_l|}}
      \prod_{j'=1}^m
      c_{j'}(\gamma)
      \int_{K_0}
      \prod_{j=1}^m
      \prod_{k=1}^{|\delta_j|} \frac{\partial X_{\delta'_{jk}}}{\partial x_{\delta_{jk}}}
      \quad\times \\
      & \quad \times \quad
      D_X^{\delta'_j} \Phi^j_{\iota_j(i,\gamma)}[\kappa_j(\gamma)]\
      \det F_K' \dX,
    \end{split}
  \end{equation}
  where we have also used Lemma~\ref{lem:changeorder} to change the
  order of multiplication and summation. Now, since the mapping~$F_K$
  is affine, the transforms $\frac{\partial X}{\partial x}$ and the
  determinant are constant and may thus be pulled out of the integral.
  As a consequence, we obtain
  \begin{equation}
    \begin{split}
      A^K_i
      &=
      \sum_{\gamma\in\mathcal{C}}
      \sum_{\delta' \in \prod_{l=1}^m [1,d]^{|\delta_l|}}
      \prod_{j'=1}^m
      c_{j'}(\gamma)
      \det F_K'
      \prod_{j''=1}^m
      \prod_{k=1}^{|\delta_{j''}|} \frac{\partial X_{\delta'_{j''k}}}{\partial x_{\delta_{j''k}}}
      \quad\times \\
      & \quad \times \quad
      \int_{K_0}
      \prod_{j=1}^m
      D_X^{\delta'_j} \Phi^j_{\iota_j(i,\gamma)}[\kappa_j(\gamma)] \dX.
    \end{split}
  \end{equation}
  The summation over $\mathcal{C}$ and $\prod_{l=1}^m
  [1,d]^{|\delta_l|}$ may now be rearranged as a summation over an
  index set $\mathcal{B}$ local to the terms of the integrand, a
  summation over an index set $\mathcal{B}'$ local to the terms
  outside of the integral, and a summation over an index set
  $\mathcal{A}$ common to all terms.
  We may thus express the element
  tensor $A^K$ as the tensor contraction
  \begin{equation}
    A^K_i = \sum_{\alpha \in \mathcal{A}} A^0_{i\alpha} G_K^{\alpha},
  \end{equation}
  where
  \begin{equation}
    \begin{split}
      A^0_{i\alpha}
      &=
      \sum_{\beta\in\mathcal{B}}
      \int_{K_0}
      \prod_{j=1}^m
      D_X^{\delta'_j(\alpha,\beta)}
      \Phi^j_{\iota_j(i,\alpha,\beta)}[\kappa_j(\alpha,\beta)]
      \dX, \\
      G_K^{\alpha}
      &=
      \prod_{j=1}^m
      c_{j}(\alpha) \,
      \det F_K'
      \sum_{\beta\in\mathcal{B'}}
      \prod_{j'=1}^m
      \prod_{k=1}^{|\delta_{j'}(\alpha,\beta)|} \frac{\partial X_{\delta'_{j'k}(\alpha,\beta)}}{\partial x_{\delta_{j'k}(\alpha,\beta)}}.
    \end{split}
  \end{equation}
  Note that since each coefficient $c_{j}(\alpha)$ in the geometry
  tensor~$G_K$ is always paired with a corresponding basis function in
  the reference tensor~$A^0$, we were able to reorder the summation to
  move the coefficients outside of the summation over~$\mathcal{B}'$.
\end{proof}

As demonstrated earlier in~\cite{logg:article:10},
the representation~(\ref{eq:representation}) in combination with
precomputation of the reference tensor $A^0$ may lead to very
efficient computation of the element tensor~$A^K$, with typical
run-time speedups ranging between a factor~$10$ and a factor~$1000$
compared to standard run-time evaluation of the element tensor by
numerical quadrature. The speedup is a direct result of the reduced
operation count for the computation of the element tensor based on the
tensor representation. In addition, one may omit multiplication with
zeros and detect symmetries or other dependencies between the entries
of the element tensor to further reduce the operation count as
discussed in~\cite{logg:article:07,logg:article:09}.

The rank of the reference tensor is determined both by the arity of
the multilinear form and how the form is expressed as a product of
coefficients and derivatives of basis functions. In general, the rank
of the reference tensor is $|i| + |\alpha|$, where $|i| = r$ is the
arity of the form and $|\alpha|$ is the rank of the geometry
tensor. As a rule of thumb, the rank of the geometry tensor is the
sum of the number of coefficients $n_C$ and the number of derivatives
$n_D$ appearing in the definition of the form, and thus the rank of
the reference tensor for a bilinear form is~$2 + n_C + n_D$.  Examples
are given below in Section~\ref{sec:benchmarks} for a set of test
cases.

\subsection{Run-time evaluation of the tensor contraction}

This framework also maps onto using matrix-vector or matrix-matrix
products at run-time.  We may recast the tensor
contraction~(\ref{eq:representation}) as a matrix-vector product for
each $K$ in the mesh.  This involves first casting $A^0$ as a matrix
by labeling all the items of the index set ${\mathcal I}$ with
integers in $[1,|{\mathcal I}|]$ and the index set ${\mathcal A}$ with
integers in $[1,|{\mathcal A}|]$. This ordering of the multiindices $i
\in \mathcal{I}$ corresponds to the rows of the matrix and the
ordering of the multiindices $\alpha \in \mathcal{A}$ corresponds to
the columns.  The same ordering is imposed on $G_K$ to make it a
vector. Furthermore, we may take a batch of elements $\mathcal{T}'
\subset \mathcal{T}$ and compute $\{A^{K} \}_{K \in \mathcal{T}'}$
with a matrix-matrix product. Currently, FFC supports code generation
that sets up the matrix-vector products via level 2 BLAS, and
computing batches of elements with matrix-matrix products via level 3
BLAS will be supported in a future version.

\subsection{Equivalence of reference tensors}

As noted earlier, forming the reference tensor is a dominant part of
the cost for FFC when compiling code for the evaluation of multilinear
forms. Before we proceed to discuss algorithms for efficiently
evaluating the tensor for a monomial term in the next section, we
conclude the discussion of form representation by noting that
particular monomial terms have the same reference tensor but different
geometry tensors. In such cases, the total cost may thus be reduced by
recognizing the common reference tensor and only computing it once.

As an example, consider each term of the two-dimensional Laplacian,
\begin{equation}
A^{K,j}_{i} = \int_{K}
\frac{\partial \phi_{i_1}^{K,1}}{\partial x_j}
\frac{\partial \phi_{i_2}^{K,2}}{\partial x_j}
\dx,
\end{equation}
where $j=1,2$ is the coordinate direction. By a suitable definition of
the index set~$\mathcal{C}$, the sum of both terms may be phrased as a
single canonical form~(\ref{eq:canonical}). We may also consider the
two terms separately and write each term in the canonical
form~(\ref{eq:canonical}). After changing coordinates to the reference
domain, we obtain the reference tensors
\begin{equation}
A^{0,j}_{i\alpha} = \int_{K_0}
\frac{\partial \Phi_{i_1}^{1}}{\partial X_{\alpha_1}}
\frac{\partial \Phi_{i_2}^{2}}{\partial X_{\alpha_2}}
\dX, \quad j=1,2,
\end{equation}
and geometry tensors
\begin{equation}
G^{\alpha}_{K,j} = \det F_K^\prime
\frac{\partial X_{\alpha_1}}{\partial x_j}
\frac{\partial X_{\alpha_2}}{\partial x_j},
\quad j=1,2.
\end{equation}
Note that the two terms of the form indeed have the same reference
tensor but different geometry tensors. This has both compile-time and
run-time implications. At compile-time, FFC should recognize this
structure, hence building the reference tensor only once and
generating code for a single geometry tensor that sums the basic
parts. When the
generated code is executed at run-time, this corresponds to fewer
instructions and hence better performance. We may formalize this as follows.
\begin{theorem}
Consider two multilinear variational forms with corresponding element
tensors $A_i^K$ and $B_i^K$ of the
form~\emph{(\ref{eq:preliminary})} defined over spaces $\{V^{A,K}_j
\}_{j=1}^{r_A}$ and $\{V^{B,K}_j\}_{j=1}^{r_B}$, respectively, that is,
\begin{equation}
  A^K_i = \int_K
  \prod_{j=1}^{r_A} D_x^{\delta^A_j} \phi^{A,K,j}_{\iota^A_j(i)} \dx,
\end{equation}
\begin{equation}
  B^K_i = \int_K
  \prod_{j=1}^{r_B} D_x^{\delta^B_j} \phi^{B,K,j}_{\iota^B_j(i)} \dx.
\end{equation}
Suppose that $r_A = r_B \equiv r$, $V^A_j = V^B_j$, $\iota^A_j =
\iota^B_j$ and $| \delta^A_j| = |\delta^B_j|$ for
$j=1,2,\ldots,r$. Then, the corresponding reference tensors are equal,
that is, $A^0 = B^0$. Moreover, this relationship is an equivalence
relation on the set of variational forms.
\end{theorem}
We remark that this result can be generalized
slightly. If permuting the integrands of one form, say $B$, would lead
to the hypotheses of this theorem being satisfied, then $A^0$ and
$B^0$ are the same after a similar permutation of the axes of $B^0$.

FFC recognizes and factors out common reference tensors by computing
for each term of a given multilinear form a string that uniquely
identifies the term. This may be accomplished by simply concatenating
the names of the finite elements that generate the function spaces for
the basis functions in the term, together with derivatives and
component indices. We refer to such a string as a
(hard)~\emph{signature} and note that the signature may be computed
cheaply for each term by just looking at its canonical
representation. We may then factor out common reference tensors by
checking for equality of signatures. If two terms have the same
signature, they also have a common reference tensor that may be
factored out.

FFC also computes a \emph{soft signature} for each term, which is
similar to a \emph{hard signature} but disregarding ordering of
multiindices. By checking for equality of soft signatures, it is
possible to find terms which have reference tensors that only differ
by the ordering of their axes. If two soft signatures match but the
corresponding hard signatures differ, it is possible to find a
reordering that results in equal hard signatures. FFC thus computes
for each term a soft signature and if the soft signatures match for
two terms, a suitable reordering is found, and the reference tensor
may be factored out as before. In Table~\ref{tab:sig}, we include the
hard and soft signatures for the bilinear form for Poisson's equation,
$a(v, u) = \int_{\Omega} \nabla v \cdot \nabla u \dx$.

\begin{table}[htbp]
  \begin{center}
    \footnotesize
    \begin{code}
    1.000000000000000e+00*
    {Lagrange finite element of degree 1 on a triangle;i0;[];[(d/dXa0)]}*
    {Lagrange finite element of degree 1 on a triangle;i1;[];[(d/dXa1)]}*dX

    1.000000000000000e+00*
    {Lagrange finite element of degree 1 on a triangle;i0;[];[(d/dXa)]}*
    {Lagrange finite element of degree 1 on a triangle;i1;[];[(d/dXa)]}*dX
    \end{code}
    \normalsize
    \caption{Hard signature (top) and soft signature (bottom) for the reference tensor of the
      bilinear form $a(v, u) = \int_{\Omega} \nabla v \cdot \nabla
      \dx$ with piecewise linear elements on triangles. Note that
      there are no line breaks in the signatures.}
    \label{tab:sig}
  \end{center}
\end{table}

%------------------------------------------------------------------------------
\section{Computing the reference tensor}
\label{sec:algorithm}

Given a multilinear variational form, the form compiler FFC
automatically generates the canonical form~(\ref{eq:canonical}) and
the representation~(\ref{eq:representation}). The computationally most
expensive part of this process is the computation of the reference
tensor~$A^0$, that is, the tabulation of each integral
\begin{equation}
  A^0_{i\alpha}
  =
  \sum_{\beta\in\mathcal{B}}
  \int_{K_0}
  \prod_{j=1}^m
  D_X^{\delta'_j(\alpha,\beta)}
  \Phi^j_{\iota_j(i,\alpha,\beta)}[\kappa_j(\alpha,\beta)]
  \dX,
\end{equation}
for $i\in\mathcal{I}$ and $\alpha\in\mathcal{A}$.

As an example, consider the bilinear form
\begin{equation} \label{eq:navierstokes,example}
  a(v, u) = \int_{\Omega} v \cdot (w \cdot \nabla) u \dx,
\end{equation}
appearing in a linearization of the incompressible Navier--Stokes
equations (see Section~\ref{sec:benchmarks} below). Computing the $12
\times 12 \times 12 \times 3 \times 3 = 15,552$ entries of the
rank~five reference tensor~$A^0$ for piecewise linear elements on
tetrahedra takes about $9.5$~seconds on a 3.0GHz Pentium~4 with FFC
version~0.2.2. Since this computation only needs to be done once at
compile-time, one may argue that this is no big issue.  However,
limitations on computer resources can be a limit to the complexity of
the forms we can compile and the degree of polynomials we can
use. Furthermore, long turn-around times to compile new, complex
models diminish the usefulness of FFC as a tool for truly rapid
development.

We present below two very different ways to compute the reference
tensor, first the obvious naive approach used in FFC version~0.2.2 and
earlier versions, and then a more efficient algorithm used in FFC
version~0.2.5 and beyond, which cuts the cost of computing the
reference tensor by several orders of magnitude. In the case of
the form~(\ref{eq:navierstokes,example}) for piecewise linear elements on
tetrahedra, the cost of computing the reference tensor is reduced from
$9.5$~seconds to around $0.02$~seconds.

\subsection{Iterating over the entries of the reference tensor}

The obvious way to compute the reference tensor is to iterate over all
indices of the reference tensor and compute each entry by quadrature
over a suitable set of quadrature points $\{X_k\}_{k=1}^{N_q}$ and a
corresponding set of quadrature weights $\{w_k\}_{k=1}^{N_q}$ on the
reference element $K_0$, as outlined in Algorithm~\ref{alg:iterating}.
Note that the iteration over multiindices $\alpha$ and $\beta$ are
themselves multiply nested loops, however the length of $\alpha,\beta$
and hence the depth of the loop nest depends on the form being
compiled. FFC uses the ``collapsed-coordinate'' Gauss-Jacobi rules
described in~\cite{KarShe99} based on taking tensor products of
Gaussian integration rules over the square and cube and mapping them
to the reference simplex. These are the arbitrary-order rules provided
by FIAT. Since we are integrating polynomials, we may pick a
quadrature rule which is exact for the total polynomial degree of the
integrand.  Alternatively, we can pick an approximate rule that is
sufficiently accurate as per the theory of variational
crimes~\cite{Cia76,BreSco94}.

\begin{algorithm}
  \begin{tabbing}
    \textbf{for}  {$i \in \mathcal{I}$}\\
    \tab \textbf{for} $\alpha \in \mathcal{A}$ \\
    \tab \tab $I = 0$ \\
    \tab \tab \textbf{for} {$\beta \in \mathcal{B}$} \\
    \tab \tab \tab \textbf{for} {$k = 1,2,\ldots,N_q$} \\
    \tab \tab \tab \tab {$I = I +
      w_k \prod_{j=1}^m
      D_X^{\delta'_j(\alpha,\beta)}
      \Phi^j_{\iota_j(i,\alpha,\beta)}[\kappa_j(\alpha,\beta)](X_k)$} \\
    \tab \tab \tab \textbf{end for} \\
    \tab \tab \textbf{end for} \\
    \tab \tab {$A^0_{i\alpha} = I$} \\
    \tab \textbf{end for} \\
    \textbf{end for}
  \end{tabbing}
  \caption{$A^0$ = ComputeReferenceTensor()}
  \label{alg:iterating}
\end{algorithm}

Algorithm~\ref{alg:iterating} is expressed at a very low granularity
--- a loop over quadrature points for each entry of the reference
tensor.  The interpretive overhead associated with this algorithm
explains the poor performance of earlier versions of FFC in a language
such as Python.  However, we may express the computation at a much
higher level of granularity and leverage optimized libraries written
in C, such as Python {\tt Numeric}~\cite{www:Numeric}. In fact, we wind up with a loop
over auxiliary indices and quadrature points, inside which the entire
reference tensor is updated by an extended outer product.  This higher
abstraction dramatically improves performance while allowing us to
remain in a high-level language.

\subsection{Assembling the reference tensor}

Algorithm~\ref{alg:iterating} may be reorganized to significantly
improve the performance. By first tabulating the basis functions at
all quadrature points according to
\begin{equation}
  \Psi^j_{k\beta,i\alpha} =
  D_X^{\delta'_j(\alpha,\beta)}
  \Phi^j_{\iota_j(i,\alpha,\beta)}[\kappa_j(\alpha,\beta)](X_k),
\end{equation}
which may be done efficiently using FIAT, we may
improve the granularity of the computation by iterating over
quadrature points $\{X_k\}_{k=1}^{N_q}$ and auxiliary
indices~$\mathcal{B}$, assembling the contributions to the reference
tensor from each pair~$(x_k, \beta)$, as outlined in
Algorithm~\ref{alg:assembling} and Algorithm~\ref{alg:product}.

\begin{algorithm}
  \begin{tabbing}
    \textbf{for} {$j = 1,2,\ldots,m$} \\
    \tab $\Psi^j = $ Tabulate($V_j^0$, $\{X_k\}_{k=1}^{N_q}$,
    $\mathcal{I}$, $\mathcal{A}$, $\mathcal{B}$,
    $\iota_j$, $\delta'_j$, $\kappa_j$) \\
    \textbf{end for} \\
    $A^0 = 0$ \\
    \textbf{for} {$k = 1,2,\ldots,N_q$} \\
    \tab \textbf{for} {$\beta \in \mathcal{B}$} \\
    \tab \tab $A^0 = A^0 \,\, +$ ComputeProduct($\{\Psi^j\}_{j=1}^m$, $k$, $\beta$) \\
    \tab \textbf{end for} \\
    \textbf{end for}
  \end{tabbing}
  \caption{$A^0$ = AssembleReferenceTensor()}
  \label{alg:assembling}
\end{algorithm}

\begin{algorithm}
  \begin{tabbing}
    $B = w_k$ \\
    \textbf{for} {$j = 1,2,\ldots,m$} \\
    \tab $B = B \otimes \Psi^j_{k\beta}$ \tab (\emph{outer product})\\
    \textbf{end for}
  \end{tabbing}
  \caption{$B$ = ComputeProduct($\{\Psi^j\}_{j=1}^m$, $k$, $\beta$)}
  \label{alg:product}
\end{algorithm}

The higher level of abstraction of Algorithm~\ref{alg:assembling}
allows us to simultaneously reduce the interpretive overhead and make
use of optimized libraries, such as the Python {\tt Numeric}
extension module.  The accumulation of the outer
products may be accomplished with the \texttt{Numeric.add} function,
which is implemented in terms of efficient C loops over the low-level
arrays.  Moreover, the sequence of outer products is accumulated
through calls to the function \texttt{Numeric.multiply.outer}. A
sketch of the Python code corresponding to
Algorithm~\ref{alg:assembling} and Algorithm~\ref{alg:product} is included in
Table~\ref{tab:assembling,python}. The full code can be downloaded
from the FFC web page~\cite{logg:www:04}.

\begin{table}[htbp]
  \normalsize
  \begin{center}
    \begin{code}
    # Iterate over quadrature points
    for k in range(num_points):
        # Iterate over secondary indices
        for beta in B:
            # Compute cumulative outer product
            P = w[k]
            for j in range(m):
                P = Numeric.multiply.outer(P, Psi[...])
            # Add to reference tensor
            Numeric.add(A0, P, A0)
    \end{code}
    \caption{A sketch of the Python implementation of
    Algorithm~\ref{alg:assembling} and Algorithm~\ref{alg:product} in FFC.}
    \label{tab:assembling,python}
  \end{center}
\end{table}

The situation is similar to that of the assembly of a global sparse
matrix for a variational form; by separating the concerns of
computing the local contribution (the element tensor) from the
insertion of the local contribution into the global sparse matrix, we
may optimize the two steps independently. In the former case, we call
the optimized code generated by FFC to compute the element tensor and
in the latter case, we may use an optimized library call such as the
PETSc~\cite{www:PETSc,BalBus04,BalGro97} call \texttt{MatSetValues()}.

Moreover, a closer investigation of Algorithm~\ref{alg:iterating} also
reveals a source of redundant computation.  As one entry in the
multiindex $\alpha$ changes, most of the factors of the product in the
innermost loop remain the same.  This problem grows worse as the arity
of the form and the number of derivatives increase.  Although this is
logically equivalent to a multiply nested loop, the structure of
looping over an enumeration of multiindices makes it highly unlikely
that an optimizing compiler would hoist invariants.  However,
Algorithm~\ref{alg:product} includes the hoisting out of the
arbitrary-depth loop nest.

%------------------------------------------------------------------------------
\section{Benchmark results}
\label{sec:benchmarks}

To measure the efficiency of the proposed
Algorithm~\ref{alg:assembling}, we compute the reference tensor for a
series of test cases, comparing FFC version~0.2.2, which is based on
Algorithm~\ref{alg:iterating}, with FFC version~0.2.5, which is based
on Algorithm~\ref{alg:assembling}.

The benchmarks were obtained on an Intel Pentium 4 (3.0~GHz CPU, 2GB
RAM) running Debian GNU/Linux with Python~2.4, Python Numeric~24.2-1
and FIAT~0.2.3. The numbers reported are the CPU times/second for the
precomputation of the reference tensor, which was previously the main
bottle-neck in the compilation of a form. With the new and more
efficient precomputation of the reference tensor in FFC, the
precomputation is no longer a bottle-neck and is in some cases
dominated by the cost of code generation.

\subsection{Test cases}
\label{sec:testcases}

We take as test cases the computation of the reference tensor for the
set of bilinear forms used to benchmark the run-time performance of
the code generated by FFC in~\cite{logg:article:10}. For convenience,
we choose a common discrete function space $V_1 = V_2 = \ldots = V_n =
V$ for all basis functions, but there is no such limitation in FFC;
function spaces can be mixed freely.

We also add a fifth test case which is a more demanding problem posing
real difficulties for earlier versions of FFC based on
Algorithm~\ref{alg:iterating}.

\subsubsection*{Test case 1: the mass matrix}

As a first test case, we consider the computation of the mass matrix
$M$ with $M_{i_1i_2} = a(\phi^1_{i_1}, \phi^2_{i_2})$ and the bilinear
form $a$ given by
\begin{equation} \label{eq:mass}
  a(v, u) = \int_{\Omega} v u \dx.
\end{equation}
The corresponding rank~two reference tensor takes the form
\begin{equation}
  A^0_{i} = \int_{K_0} \Phi_{i_1} \Phi_{i_2} \dX,
\end{equation}
with the rank~zero geometry tensor given by $G_K = \det F_K'$.

\subsubsection*{Test case 2: Poisson's equation}

As a second example, consider the bilinear form for Poisson's equation,
\begin{equation} \label{eq:poisson}
  a(v, u) = \int_{\Omega} \nabla v \cdot \nabla u \dx.
\end{equation}
The corresponding rank~four reference tensor takes the form
\begin{equation}
  A^0_{i\alpha} = \int_{K_0}
  \frac{\partial \Phi_{i_1}}{\partial X_{\alpha_1}}
  \frac{\partial \Phi_{i_2}}{\partial X_{\alpha_2}} \dX,
\end{equation}
with the rank~two geometry tensor given by
$G_K^{\alpha} = \det F_K' \sum_{\beta=1}^d
\frac{\partial X_{\alpha_1}}{\partial x_{\beta}}
\frac{\partial X_{\alpha_2}}{\partial x_{\beta}}$.

\subsubsection*{Test case 3: Navier--Stokes}

We consider next the nonlinear term $u \cdot \nabla u$ of the
incompressible Navier--Stokes equations,
\begin{equation}
  \begin{split}
    \dot{u} + u \cdot \nabla u - \nu \Delta u + \nabla p &= f, \\
    \nabla \cdot u &= 0.
  \end{split}
\end{equation}
Linearizing this term as part of either a Newton or fixed-point based
solution method (see for example~\cite{EriEst03c,HofJoh04b}),we need
to evaluate the bilinear form
\begin{equation}
  a(v, u) = a_w(v, u) = \int_{\Omega} v \cdot (w \cdot \nabla) u \dx.
\end{equation}
The corresponding rank~five reference tensor takes the form
\begin{equation}
  A^0_{i\alpha} =
  \sum_{\beta=1}^d
  \int_{K_0}
  \Phi_{i_1}[\beta] \Phi_{\alpha_1}[\alpha_2]
  \frac{\partial \Phi_{i_2}[\beta]}{\partial X_{\alpha_3}}
  \dX,
\end{equation}
with the rank~three geometry tensor given by
$G_K^{\alpha} = \det F_K' w^K_{\alpha_1}
\frac{\partial X_{\alpha_3}}{\partial x_{\alpha_2}}$.

\subsubsection*{Test case 4: linear elasticity}

As our next test case, we consider the strain-strain term of linear
elasticity~\cite{BreSco94},
\begin{equation}
  \begin{split}
    a(v, u) &= \int_{\Omega} \frac{1}{4}
    (\nabla v + (\nabla v)^{\top}) : (\nabla u + (\nabla u)^{\top}) \dx \\
    &= \int_{\Omega}
    \sum_{i,j=1}^d
    \frac{1}{4}
    \frac{\partial v_i}{\partial x_j}
    \frac{\partial u_i}{\partial x_j} +
    \frac{1}{4}
    \frac{\partial v_i}{\partial x_j}
    \frac{\partial u_j}{\partial x_i} +
    \frac{1}{4}
    \frac{\partial v_j}{\partial x_i}
    \frac{\partial u_i}{\partial x_j} +
    \frac{1}{4}
    \frac{\partial v_j}{\partial x_i}
    \frac{\partial u_j}{\partial x_i} \dx \\
    &= \int_{\Omega}
    \sum_{i,j=1}^d
    \frac{1}{2}
    \frac{\partial v_i}{\partial x_j}
    \frac{\partial u_i}{\partial x_j} +
    \frac{1}{2}
    \frac{\partial v_i}{\partial x_j}
    \frac{\partial u_j}{\partial x_i} \dx.
  \end{split}
\end{equation}
Considering here for simplicity only the first of the
two terms\footnote{The benchmark
measures the time to compute the reference tensors
for both terms.}, the rank~four reference tensor takes
the form
\begin{equation}
  A^0_{i\alpha} =
  \sum_{\beta=1}^d
  \int_{K_0}
  \frac{\partial \Phi_{i_1}[\beta]}{\partial X_{\alpha_1}}
  \frac{\partial \Phi_{i_2}[\beta]}{\partial X_{\alpha_2}}
  \dX,
\end{equation}
with the rank~two geometry tensor given by
$G_{K}^{\alpha} =
\frac{1}{2}
\det F_K'
\sum_{\beta=1}^d
\frac{\partial X_{\alpha_1}}{\partial x_{\beta}}
\frac{\partial X_{\alpha_2}}{\partial x_{\beta}}$.

\subsubsection*{Test case 5: stabilization}

As a final test case, we consider the bilinear form for a
stabilization term appearing in a least-squares stabilized
$\mathrm{cG}(1)\mathrm{cG}(1)$ method for the incompressible
Navier--Stokes equations~\cite{EriEst03c,HofJoh04b}:
\begin{equation}
  a(v, u) = \int_{\Omega}
  (w \cdot \nabla v) \cdot
  (w \cdot \nabla u) \dx
  =
  \sum_{i,j,k=1}^d
  w[j] \frac{\partial v[i]}{\partial x_j}
  w[k] \frac{\partial u[i]}{\partial x_k} \dx.
\end{equation}
The corresponding rank~eight reference tensor takes the form
\begin{equation}
  A^0_{i\alpha} =
  \sum_{\beta=1}^d
  \int_{K_0}
  \Phi_{\alpha_1}[\alpha_3]
  \frac{\partial \Phi_{i_1}[\beta]}{\partial X_{\alpha_5}}
  \Phi_{\alpha_2}[\alpha_4]
  \frac{\partial \Phi_{i_2}[\beta]}{\partial X_{\alpha_6}}
  \dX,
\end{equation}
with the rank~six geometry tensor given by
$G_K^{\alpha} =
w^K_{\alpha_1} w^K_{\alpha_2} \,
\det F_K'
\frac{\partial X_{\alpha_5}}{\partial x_{\alpha_3}}
\frac{\partial X_{\alpha_6}}{\partial x_{\alpha_4}}$.
As a consequence of the high rank of the reference tensor, the
computation of the reference tensor is very costly. For piecewise
linear basis functions on tetrahedra with $4 \times 3 = 12$ basis
functions on the reference element, the number of entries in the
reference tensor is
$12 \times 12 \times 12 \times 12 \times 3 \times 3 \times 3 \times 3
= 1,679,616$.

\subsection{Results}

In Table~\ref{tab:speedup}, we present a summary of the speedups
obtained with Algorithm~\ref{alg:assembling} (FFC version~0.2.5)
compared to Algorithm~\ref{alg:iterating} (FFC version~0.2.2).
Detailed results are given in
Figures~\ref{fig:result,1}--\ref{fig:result,4} for test cases~1--4.
Because of limitations in the earlier version of FFC, that is,
the poor performance of Algorithm~\ref{alg:iterating}, the comparison
is made for polynomial degree $q \leq 8$ in test cases~1--2, $q \leq 3$
in test cases~3--4 and $q = 1$ in test case~5. Higher degree forms may
be compiled with FFC version~0.2.5, but even then the memory
requirements for storing the reference tensor may in some cases exceed
the available 2GB on the test system.

As evident from Table~\ref{tab:speedup}, the speedup is significant
in most test cases, typically one or two orders of magnitude.
In test case~5, the stabilization term in Navier-Stokes, the speedup
is as large as three orders of magnitude.

\begin{table}[htbp]
  \begin{center}
    \footnotesize
    \begin{tabular}{|l|r|r|r|r|r|r|r|r|}
      \hline
      Test case & $q = 1$ & $q = 2$ & $q = 3$ & $q = 4$ & $q = 5$ & $q = 6$ & $q = 7$ & $q = 8$ \\
      \hline
      \hline
      1. Mass matrix 2D & 1.4 & 2.6 & 4.0 & 5.6 & 7.6 & 9.9 & 12.5 & 15.2 \\
      1. Mass matrix 3D & 1.6 & 3.5 & 6.4 & 10.8 & 16.9 & 23.1 & 28.3 & 20.9 \\
      2. Poisson 2D & 2.5 & 7.0 & 11.4 & 16.4 & 21.9 & 27.5 & 33.5 & 39.4 \\
      2. Poisson 3D & 7.4 & 19.3 & 33.8 & 47.8 & 43.8 & 38.8 & 28.1 & 23.1 \\
      3. Navier--Stokes 2D & 67.2 & 264.3 & 239.0 & --- & --- & --- & --- & --- \\
      3. Navier--Stokes 3D & 461.3 & 291.7 & 82.3 & --- & --- & --- & --- & --- \\
      4. Elasticity 2D & 20.2 & 44.3 & 68.9 & --- & --- & --- & --- & --- \\
      4. Elasticity 3D & 142.5 & 230.7 & 138.0 & --- & --- & --- & --- & --- \\
      5. Stabilization 2D & 1114.7 & --- & --- & --- & --- & --- & --- & --- \\
      5. Stabilization 3D & 1101.4 & --- & --- & --- & --- & --- & --- & --- \\
      \hline
    \end{tabular}
    \normalsize
    \caption{Speedups for test cases 1--5 in 2D and 3D for different polynomial degrees~$q$ of Lagrange basis functions.}
    \label{tab:speedup}
  \end{center}
\end{table}

\begin{figure}[htbp]
  \begin{center}
    \psfrag{q}{$q$}
    \includegraphics[width=11cm]{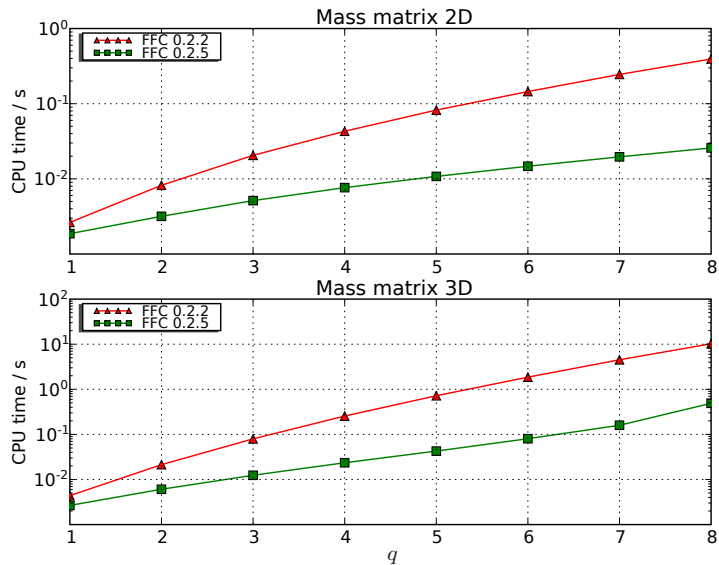}
    \caption{Compilation time as function of polynomial degree~$q$ for test case~1, the mass matrix,
      specified in FFC by \texttt{a = v*u*dx}.}
    \label{fig:result,1}
  \end{center}
\end{figure}

\begin{figure}[htbp]
  \begin{center}
    \psfrag{q}{$q$}
    \includegraphics[width=11cm]{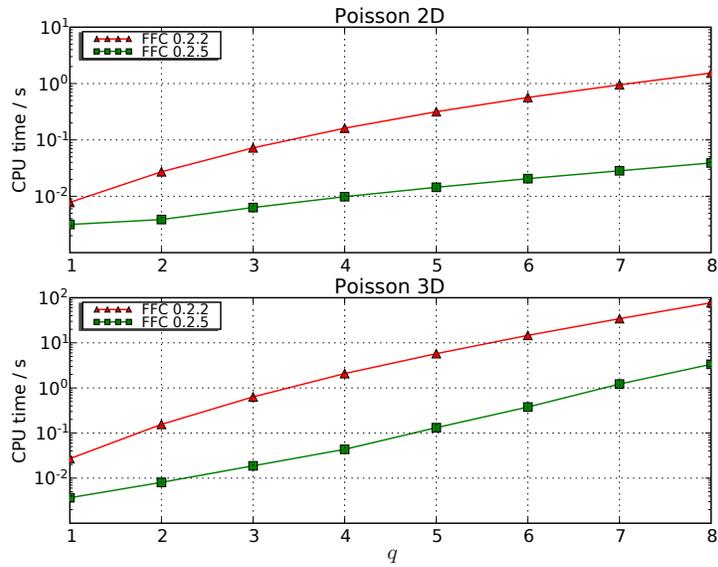}
    \caption{Compilation time as function of polynomial degree~$q$ for test case~2, Poisson's equation,
      specified in FFC by \texttt{a = v.dx(i)*u.dx(i)*dx}.}
    \label{fig:result,2}
  \end{center}
\end{figure}

\begin{figure}[htbp]
  \begin{center}
    \psfrag{q}{$q$}
    \includegraphics[width=11cm]{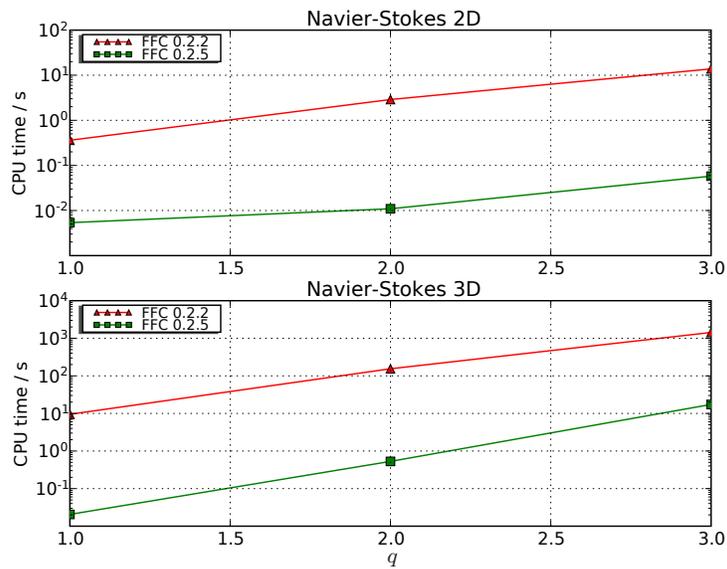}
    \caption{Compilation time as function of polynomial degree~$q$ for test case~3, the nonlinear term of
      the incompressible Navier--Stokes equations, specified in FFC
      by \texttt{a = v[i]*w[j]*u[i].dx(j)*dx}.}
    \label{fig:result,3}
  \end{center}
\end{figure}

\begin{figure}[htbp]
  \begin{center}
    \psfrag{q}{$q$}
    \includegraphics[width=11cm]{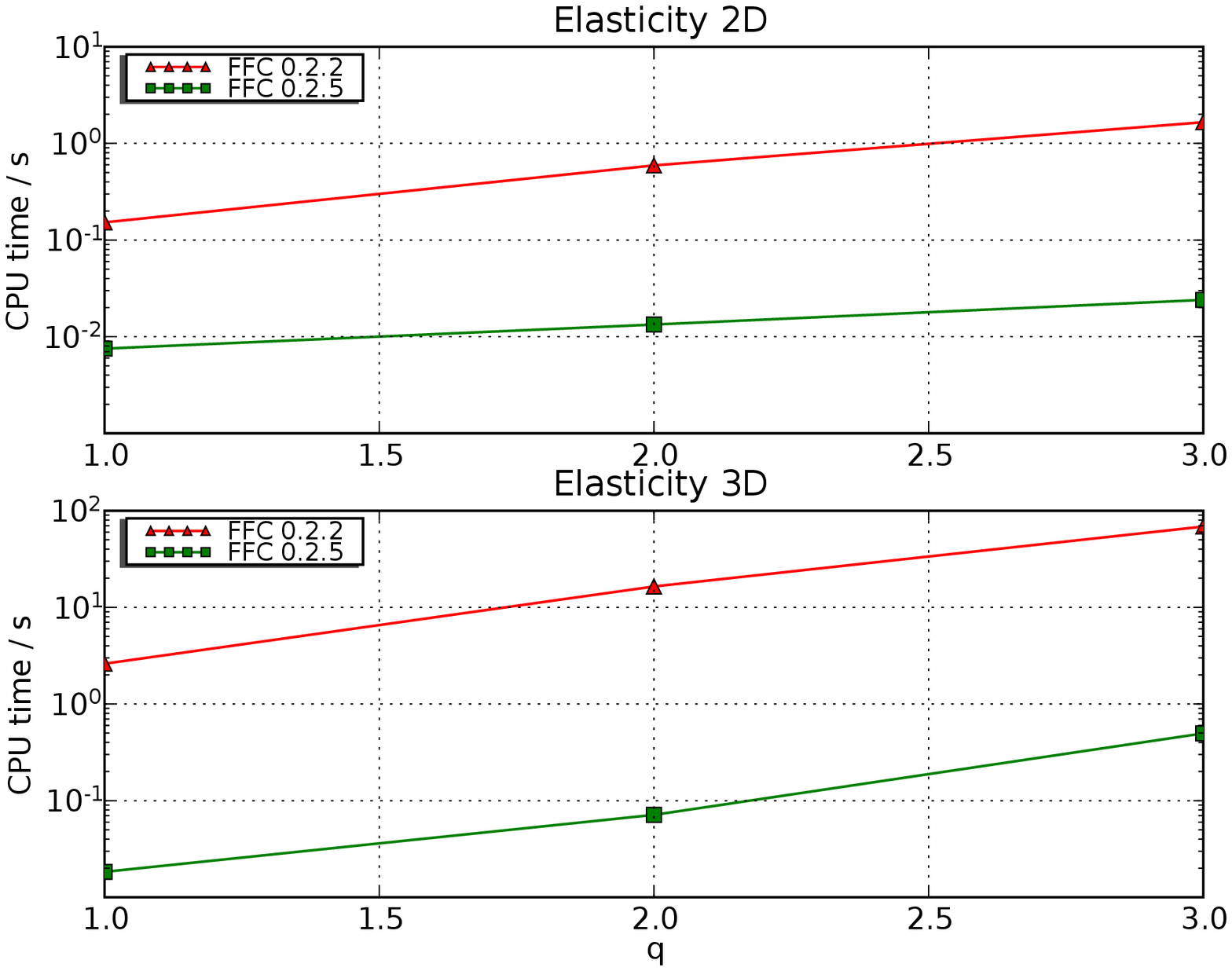}
    \caption{Compilation time as function of polynomial degree~$q$ for test case~4, the strain-strain term
      of linear elasticity, specified in FFC
      by \texttt{a = 0.25*(v[i].dx(j) + v[j].dx(i)) * (u[i].dx(j) + u[j].dx(i)) * dx}.}
    \label{fig:result,4}
  \end{center}
\end{figure}

%------------------------------------------------------------------------------
\section{Conclusion}
\label{sec:conclusion}

The new improved precomputation of the reference tensor removes an
outstanding bottleneck in the compilation of variational forms. This
improves the possibilities of using FFC as a tool for rapid
prototyping and development. The feature set for FFC is also quickly
expanding, with an expanded form language, recently added support for
arbitrary mixed formulations, and with built-in support for functionals,
nonlinear formulations and error estimates on the horizon. At the same
time, FFC is still very much a test-bed for basic research in
efficient evaluation of general variational forms.

%------------------------------------------------------------------------------
\begin{ack}

We wish to thank Johan Hoffman, Johan Jansson, Matthew Knepley, Ola
Skavhaug, Andy Terrel and Garth~N.~Wells for testing early versions of
the compiler and providing constructive feedback and real-world
problems that motivated the development of the new improved tabulation
of integrals.

\end{ack}
%------------------------------------------------------------------------------
\bibliographystyle{acmtrans}
\bibliography{bibliography}

\begin{thebibliography}{}

\bibitem[\protect\citeauthoryear{Bagheri and Scott}{Bagheri and
  Scott}{2003}]{www:Analysa}
{\sc Bagheri, B.} {\sc and} {\sc Scott, L.~R.} 2003.
\newblock Analysa.
\newblock URL: \url{http://people.cs.uchicago.edu/~ridg/al/aa.html}.

\bibitem[\protect\citeauthoryear{Balay, Buschelman, Eijkhout, Gropp, Kaushik,
  Knepley, McInnes, Smith, and Zhang}{Balay et~al\mbox{.}}{2004}]{BalBus04}
{\sc Balay, S.}, {\sc Buschelman, K.}, {\sc Eijkhout, V.}, {\sc Gropp, W.~D.},
  {\sc Kaushik, D.}, {\sc Knepley, M.~G.}, {\sc McInnes, L.~C.}, {\sc Smith,
  B.~F.}, {\sc and} {\sc Zhang, H.} 2004.
\newblock {PETS}c users manual.
\newblock Tech. Rep. ANL-95/11 - Revision 2.1.5, Argonne National Laboratory.

\bibitem[\protect\citeauthoryear{Balay, Buschelman, Gropp, Kaushik, Knepley,
  McInnes, Smith, and Zhang}{Balay et~al\mbox{.}}{2006}]{www:PETSc}
{\sc Balay, S.}, {\sc Buschelman, K.}, {\sc Gropp, W.~D.}, {\sc Kaushik, D.},
  {\sc Knepley, M.~G.}, {\sc McInnes, L.~C.}, {\sc Smith, B.~F.}, {\sc and}
  {\sc Zhang, H.} 2006.
\newblock {PETS}c.
\newblock URL: \url{http://www.mcs.anl.gov/petsc/}.

\bibitem[\protect\citeauthoryear{Balay, Gropp, McInnes, and Smith}{Balay
  et~al\mbox{.}}{1997}]{BalGro97}
{\sc Balay, S.}, {\sc Gropp, W.~D.}, {\sc McInnes, L.~C.}, {\sc and} {\sc
  Smith, B.~F.} 1997.
\newblock Efficient management of parallelism in object oriented numerical
  software libraries.
\newblock In {\em Modern Software Tools in Scientific Computing}, {E.~Arge},
  {A.~M. Bruaset}, {and} {H.~P. Langtangen}, Eds. Birkh{\"{a}}user Press,
  163--202.

\bibitem[\protect\citeauthoryear{Bangerth, Hartmann, and Kanschat}{Bangerth
  et~al\mbox{.}}{2006}]{www:deal.II}
{\sc Bangerth, W.}, {\sc Hartmann, R.}, {\sc and} {\sc Kanschat, G.} 2006.
\newblock {\tt deal.{I}{I}} {D}ifferential {E}quations {A}nalysis {L}ibrary.
\newblock URL: \url{http://www.dealii.org/}.

\bibitem[\protect\citeauthoryear{Brenner and Scott}{Brenner and
  Scott}{1994}]{BreSco94}
{\sc Brenner, S.~C.} {\sc and} {\sc Scott, L.~R.} 1994.
\newblock {\em The Mathematical Theory of Finite Element Methods}.
\newblock Springer-Verlag.

\bibitem[\protect\citeauthoryear{Ciarlet}{Ciarlet}{1976}]{Cia76}
{\sc Ciarlet, P.~G.} 1976.
\newblock {\em Numerical Analysis of the Finite Element Method}.
\newblock Les Presses de l'Universite de Montreal.

\bibitem[\protect\citeauthoryear{Dular and Geuzaine}{Dular and
  Geuzaine}{2006}]{www:GetDP}
{\sc Dular, P.} {\sc and} {\sc Geuzaine, C.} 2006.
\newblock Get{DP}: a {G}eneral environment for the treatment of {D}iscrete
  {P}roblems.
\newblock URL: \url{http://www.geuz.org/getdp/}.

\bibitem[\protect\citeauthoryear{Eriksson, Estep, Hansbo, and Johnson}{Eriksson
  et~al\mbox{.}}{1996}]{EriEst96}
{\sc Eriksson, K.}, {\sc Estep, D.}, {\sc Hansbo, P.}, {\sc and} {\sc Johnson,
  C.} 1996.
\newblock {\em Computational Differential Equations}.
\newblock Cambridge University Press.

\bibitem[\protect\citeauthoryear{Eriksson, Estep, and Johnson}{Eriksson
  et~al\mbox{.}}{2003}]{EriEst03c}
{\sc Eriksson, K.}, {\sc Estep, D.}, {\sc and} {\sc Johnson, C.} 2003.
\newblock {\em Applied Mathematics: Body and Soul}. Vol. III.
\newblock Springer-Verlag.

\bibitem[\protect\citeauthoryear{Hoffman, Jansson, Johnson, Knepley, Kirby,
  Logg, Scott, and Wells}{Hoffman et~al\mbox{.}}{2006}]{logg:www:03}
{\sc Hoffman, J.}, {\sc Jansson, J.}, {\sc Johnson, C.}, {\sc Knepley, M.~G.},
  {\sc Kirby, R.~C.}, {\sc Logg, A.}, {\sc Scott, L.~R.}, {\sc and} {\sc Wells,
  G.~N.} 2006.
\newblock {\em {FE}ni{CS}}.
\newblock \url{http://www.fenics.org/}.

\bibitem[\protect\citeauthoryear{Hoffman, Jansson, Logg, and Wells}{Hoffman
  et~al\mbox{.}}{2006}]{logg:www:01}
{\sc Hoffman, J.}, {\sc Jansson, J.}, {\sc Logg, A.}, {\sc and} {\sc Wells,
  G.~N.} 2006.
\newblock {\em {DOLFIN}}.
\newblock \url{http://www.fenics.org/dolfin/}.

\bibitem[\protect\citeauthoryear{Hoffman and Johnson}{Hoffman and
  Johnson}{2004}]{HofJoh04b}
{\sc Hoffman, J.} {\sc and} {\sc Johnson, C.} 2004.
\newblock {\em Encyclopedia of Computational Mechanics, Volume 3, Chapter 7:
  Computability and Adaptivity in {CFD}}.
\newblock Wiley.

\bibitem[\protect\citeauthoryear{Hughes}{Hughes}{1987}]{Hug87}
{\sc Hughes, T. J.~R.} 1987.
\newblock {\em The Finite Element Method: Linear Static and Dynamic Finite
  Element Analysis}.
\newblock Prentice-Hall.

\bibitem[\protect\citeauthoryear{Karniadakis and Sherwin}{Karniadakis and
  Sherwin}{1999}]{KarShe99}
{\sc Karniadakis, G.~E.} {\sc and} {\sc Sherwin, S.~J.} 1999.
\newblock {\em Spectral/{$hp$} element methods for {CFD}}.
\newblock Numerical Mathematics and Scientific Computation. Oxford University
  Press, New York.

\bibitem[\protect\citeauthoryear{Kirby}{Kirby}{2004}]{Kir04}
{\sc Kirby, R.~C.} 2004.
\newblock {FIAT}: A new paradigm for computing finite element basis functions.
\newblock {\em ACM Trans. Math. Software\/}~{\em 30}, 502--516.

\bibitem[\protect\citeauthoryear{Kirby}{Kirby}{2006a}]{www:FIAT}
{\sc Kirby, R.~C.} 2006a.
\newblock {\em {FIAT}}.
\newblock URL: \url{http://www.fenics.org/fiat/}.

\bibitem[\protect\citeauthoryear{Kirby}{Kirby}{2006b}]{Kir05}
{\sc Kirby, R.~C.} 2006b.
\newblock Optimizing {FIAT} with {Level 3 BLAS}.
\newblock {\em to appear in {ACM} Transactions on Mathematical Software\/}.

\bibitem[\protect\citeauthoryear{Kirby, Knepley, Logg, and Scott}{Kirby
  et~al\mbox{.}}{2005}]{logg:article:07}
{\sc Kirby, R.~C.}, {\sc Knepley, M.~G.}, {\sc Logg, A.}, {\sc and} {\sc Scott,
  L.~R.} 2005.
\newblock Optimizing the evaluation of finite element matrices.
\newblock {\em SIAM J. Sci. Comput.\/}~{\em 27,\/}~3, 741--758.

\bibitem[\protect\citeauthoryear{Kirby, Knepley, and Scott}{Kirby
  et~al\mbox{.}}{2004}]{KirKne04}
{\sc Kirby, R.~C.}, {\sc Knepley, M.~G.}, {\sc and} {\sc Scott, L.~R.} 2004.
\newblock Evaluation of the action of finite element operators.
\newblock Tech. Rep. TR--2004--07, University of Chicago, Department of
  Computer Science.

\bibitem[\protect\citeauthoryear{Kirby and Logg}{Kirby and
  Logg}{2006}]{logg:article:10}
{\sc Kirby, R.~C.} {\sc and} {\sc Logg, A.} 2006.
\newblock A compiler for variational forms.
\newblock {\em to appear in {ACM} Trans. Math. Softw.\/}.

\bibitem[\protect\citeauthoryear{Kirby, Logg, Scott, and Terrel}{Kirby
  et~al\mbox{.}}{2006}]{logg:article:09}
{\sc Kirby, R.~C.}, {\sc Logg, A.}, {\sc Scott, L.~R.}, {\sc and} {\sc Terrel,
  A.~R.} 2006.
\newblock Topological optimization of the evaluation of finite element
  matrices.
\newblock {\em {SIAM} J. Sci. Comput.\/}~{\em 28,\/}~1, 224--240.

\bibitem[\protect\citeauthoryear{Langtangen}{Langtangen}{1999}]{Lan99}
{\sc Langtangen, H.~P.} 1999.
\newblock {\em Computational Partial Differential Equations -- Numerical
  Methods and Diffpack Programming}.
\newblock Lecture Notes in Computational Science and Engineering. Springer.

\bibitem[\protect\citeauthoryear{Logg}{Logg}{2004}]{logg:thesis:03}
{\sc Logg, A.} 2004.
\newblock Automation of computational mathematical modeling.
\newblock Ph.D. thesis, Chalmers University of Technology, Sweden.

\bibitem[\protect\citeauthoryear{Logg}{Logg}{2006}]{logg:www:04}
{\sc Logg, A.} 2006.
\newblock {\em {FFC}}.
\newblock \url{http://www.fenics.org/ffc/}.

\bibitem[\protect\citeauthoryear{Long}{Long}{2003}]{Lon03}
{\sc Long, K.} 2003.
\newblock Sundance, a rapid prototyping tool for parallel {PDE}-constrained
  optimization.
\newblock In {\em Large-Scale PDE-Constrained Optimization}. Lecture notes in
  computational science and engineering. Springer-Verlag.

\bibitem[\protect\citeauthoryear{Oliphant et~al\mbox{.}}{Oliphant
  et~al\mbox{.}}{2006}]{www:Numeric}
{\sc Oliphant, T.} {\sc et~al\mbox{.}} 2006.
\newblock {\em {Python Numeric}}.
\newblock URL: \url{http://numeric.scipy.org/}.

\bibitem[\protect\citeauthoryear{Pironneau, Hecht, Hyaric, and
  Ohtsuka}{Pironneau et~al\mbox{.}}{2006}]{www:FreeFEM}
{\sc Pironneau, O.}, {\sc Hecht, F.}, {\sc Hyaric, A.~L.}, {\sc and} {\sc
  Ohtsuka, K.} 2006.
\newblock Free{FEM}.
\newblock URL: \url{http://www.freefem.org/}.

\end{thebibliography}

\begin{received}
\end{received}

\end{document}